\documentclass[reqno]{amsart}
\usepackage{amssymb}

\usepackage{tikz}
\usetikzlibrary{decorations.pathreplacing}
\usetikzlibrary{arrows}

\usepackage{color} 

\usepackage{appendix}

 \usepackage{abstract}

\newtheorem{theorem}{Theorem}[section]
\newtheorem{corollary}[theorem]{Corollary}
\newtheorem{proposition}[theorem]{Proposition}
\newtheorem{lemma}[theorem]{Lemma}
\newtheorem*{conjecture}{Conjecture}

\theoremstyle{definition}
\newtheorem{remark}[theorem]{Remark}

\theoremstyle{definition}
\newtheorem{definition}[theorem]{Definition}
\newtheorem{criterion}[theorem]{Criterion}

\newcommand{\U}{{\mathcal U}}
\newcommand{\Hil}{{\mathcal H}}
\newcommand{\ep}{{\varepsilon}}
\newcommand{\dvol}{\rm dvol}
\newcommand{\Ker}{\rm Ker}
\newcommand{\Spec}{\rm Spec}

\begin{document}

\title[Multiplicity of eigenvalues of Conf. Cov. operators]
{On the multiplicity of eigenvalues of conformally covariant operators}

\author[Y. Canzani]{Yaiza Canzani}
\address{Department of Mathematics and
Statistics, McGill University, 805 Sherbrooke St. West, Montr\'eal
QC H3A 2K6, Ca\-na\-da.} \email{canzani@math.mcgill.ca}


\thanks{*The author was supported by a Schulich Graduate Fellowship.}

\maketitle

\begin{center}
\emph{To appear in Annales de L'Institut Fourier.}
\end{center}\ \\

\begin{abstract}
Let $(M,g)$ be a compact Riemannian manifold and  $P_g$ an elliptic, formally self-adjoint, conformally 
covariant operator  acting on smooth sections of a bundle over $M$.
We prove that if $P_g$ \emph{has no rigid eigenspaces} (see Definition \ref{rigid eigenspace}),
the set of functions  $f\in C^\infty(M, \mathbb R)$ for which $P_{e^fg}$ has only simple non-zero eigenvalues  is a residual
 set in $C^\infty(M,\mathbb R)$. As a consequence we prove that 
  if $P_g$ \emph{has no rigid eigenspaces} for a dense set of metrics, then  all non-zero eigenvalues are simple for a residual set of metrics in the $C^\infty$-topology.
We also prove that the eigenvalues of $P_g$ depend continuously on $g$ in the $C^\infty$-topology, provided $P_g$ is strongly elliptic.
As an application of our work,  we show that if $P_g$  acts on $C^\infty(M)$
(e.g. GJMS operators), its non-zero eigenvalues are generically simple. 
\end{abstract}

\section{Introduction}

Conformally covariant operators (see Definition \ref{cco}) are known to play a key role in 
Physics and Spectral Geometry.
In the past few years, much work has been done on their systematic construction, 
understanding, and classification \cite{Bas,Bra,Bra2,Bra3,Eas,Gov,GJMS,Pan,Wun}.
In Physics, the interest for conformally covariant operators started when Bateman \cite{Bat}
discovered that the classical field equations describing massless particles
(like Maxwell and Dirac equations) depend only on the conformal structure.
These operators are also important tools in String Theory and 
Quantum Gravity, used to relate scattering matrices on conformally compact Einstein manifolds with 
conformal objects on their boundaries at infinity \cite{GZ}.
In Spectral Geometry, the purpose is to relate global geometry to
the spectrum of some natural operators on the manifold.
For example, the nice behavior of conformally invariant operators with respect to conformal deformations
of a metric yields a closed expression for the conformal variation
of the determinants leading to important progress in the lines of \cite{Bra, BCY, BO}.\\

When it comes to perturbing a metric to deal with any of the problems described above, 
it is often very helpful and simplifying to work under the assumption that the eigenvalues of a given operator
are a smooth, or even continuous,  function of a metric perturbation parameter. But reality is much more complicated,
and  usually, when possible, one has to find indirect ways of arriving to the desired results without such assumption.
However, it is generally believed that eigenvalues of formally self-adjoint  operators with positive leading symbol 
on $SO(m)$ or $Spin(m)$ irreducible bundles are generically simple. And, as Branson and
{\O}rsted point out in \cite[pag 22]{BO2}, since many of the quantities of interest are universal expressions,
the generic case is often all that one needs.
In many cases, it has been shown that the eigenvalues of  \emph{metric dependent}, formally self-adjoint,
elliptic operators are generically simple. 
The main example is the \emph{Laplace} operator on smooth functions on a compact manifold, see \cite{Uhl,Tey,BU,BW}. 
The simplicity of eigenvalues has also been shown, generically, for the \emph{Hodge-Laplace} operator on forms on a compact manifold of  dimension $3$ (see \cite{EP}). 
Besides, in 2002, Dahl proved such result for the Dirac operator on spinors of a compact spin
manifold of dimension $3$; see  \cite{Dah}.
It seems to be the case that in the class of conformally covariant operators the latter is the only situation for
which the simplicity of the eigenvalues has generically been established.
In this note we hope to shed some light in this direction.\\

A summary of the main results follows.
Let $(M,g)$ be a compact Riemannian manifold and $E_g$ a smooth bundle over $M$.
Consider  $P_g:\Gamma(E_g) \to \Gamma(E_g)$ to be an elliptic, formally self-adjoint, conformally 
covariant operator  of order $m$ acting on smooth sections of the bundle $E_g$. Endow the space $\mathcal M$
of Riemannian metrics over $M$ with the $C^\infty$-topology. Among the main results are: \medskip

\begin{list}{$-$}{\leftmargin=1em \itemsep=1em}

\item Suppose $P_g:\Gamma(E_g) \to \Gamma(E_g)$ \emph{has no rigid eigenspaces} (see Definition \ref{rigid eigenspace}).
Then the set of functions  $f\in C^\infty(M, \mathbb R)$ for which $P_{e^fg}$ has only simple non-zero eigenvalues  is a residual
 set in $C^\infty(M,\mathbb R)$. As a corollary we prove that if $P_g$ has no rigid eigenspaces for a dense set of metrics,
then all non-zero eigenvalues are simple for a residual set of metrics in $\mathcal M$.
 
\item Suppose $P_g:\Gamma(E_g) \to \Gamma(E_g)$ satisfies the unique continuation principle for a dense set of metrics in $\mathcal M$. Then
 the multiplicity of all non-zero eigenvalues is smaller than the rank of the bundle for a residual set of metrics in $\mathcal M$.

\item As an application, if $P_g$  acts on $C^\infty(M)$
(e.g. GJMS operator), its non-zero eigenvalues are simple for a residual set of metrics in $\mathcal M$. 

\item If
$P_g:\Gamma(E_g) \to \Gamma(E_g)$ is strongly elliptic,
 then the eigenvalues of $P_g$ depend continuously on $g$ in the $C^\infty$-topology of metrics.

\end{list}\ \\ 
Not many statements can be proved simultaneously for all conformally covariant operators, even if self-adjointness and ellipticity are enforced. Some of these operators act on functions, others act on bundles.  For some of them the maximum principle is satisfied, whereas for others is not. Some of them are bounded below while others are not. We would therefore like to emphasize  that we find remarkable that our techniques work for the whole class of conformally covariant operators.

\subsection*{Acknowledgements}\ \\  \indent One of the mains results of the paper, Corollary \ref{main cor}, is a generalization to the whole class of conformally covariant operators of the results presented by M. Dahl in \cite{Dah}, for the Dirac operator in $3$-manifolds. 
In an earlier version of this manuscript, part of the argument in M. Dahl's paper was being reproduced (namely, the first two lines of the proof of  Proposition 3.2 in \cite{Dah}). At the end of June 2012,  the author started working with Rapha\"el Ponge
 on an extension of the results in this paper to the class of pseudodifferential operators. While doing so,  R. Ponge realized
there was a mistake in M. Dahl's argument which was reproduced in an earlier version of this manuscript. The author is grateful to Rapha\"el Ponge for pointing out the mistake and for providing useful suggestions on the manuscript. \ \\ 
\indent It is a great pleasure to thank Rod Gover, Dmitry Jakobson and Niky Kamran  for detailed comments on the manuscript and many helpful discussions. 
\section{Statement of the results}

In order to provide a precise description of our results, we introduce the following\\

 { \bf Conventions}. Let $(M,g)$  denote a compact Riemannian manifold and consider a
  smooth bundle $E_g$ over $M$ with  product on the fibers  $(\, , \,)_x$.
  Write $\Gamma(E_g)$ for the space of smooth sections and denote by $\langle\,,\, \rangle_g$ 
   the global inner product
$\langle u,v \rangle_g=\int_M (u(x),v(x))_x \dvol_g$, for $u,v \in \Gamma(E_g)$. \\
 A differential operator $P_g:\Gamma(E_g) \to \Gamma(E_g)$ is said to  be \emph{formally self-adjoint} if for
 all $u,v \in \Gamma(E_g)$ we have $\langle P_gu, v \rangle_g= \langle u, P_gv \rangle_g$. Let $\sigma_{P_g}$ 
 denote the principal symbol of $P_g$ and let $m$ be the order of $P_g$.
  We say that $P_g$ is \emph{elliptic} if  
$\sigma_{P_g}(\xi):(E_g)_x \to (E_g)_x$ is an invertible map for all $(x,\xi) \in T^*M$, $\xi \neq 0$.
For a definition of a conformally covariant operator see Definition \ref{cco}.\\

Throughout this paper we work under the following assumptions:\\

\begin{itemize}
\item $M$ is a compact differentiable manifold, $g$ is a Riemannian metric over $M$
and $E_g$ denotes a smooth bundle over $M$ as described above. \\

\item  {\bf $P_g:\Gamma(E_g) \to \Gamma(E_g)$ is an elliptic, formally self-adjoint,
 conformally covariant operator of order $m$}.\\
 
\item The space $\mathcal M$ of Riemannian metrics over $M$, is endowed with the $C^\infty$-topology:
Fix a background metric  $g$, and define the distance $d^m_g$ between two metrics $ g_1,\, g_2$ by 
$d^m_g(g_1,g_2):=\max_{k=0,\dots,m} \| \nabla_g^k \,(g_1-g_2)\|_\infty.$
The topology induced by $d_g^m$ is independent of the background metric and it is called
the $C^m$-topology of metrics on $M$.

\end{itemize}\ \\

There are many ways of splitting the spectrum of an operator. The main ideas in this paper are inspired by the constructive methods of Bleecker
and  Wilson \cite{BW}.
In what follows the main results of this paper are stated.

\begin{theorem}\label{main thm}
For $P_g:C^\infty(M) \to C^\infty(M)$, 
 the set of functions $f\in C^\infty(M,\mathbb R)$ for which all the non-zero eigenvalues of $P_{e^fg}$
are simple is a residual set in $C^\infty(M, \mathbb R)$.
\end{theorem} \smallskip

To obtain a generalization of Theorem \ref{main thm} for operators acting on bundles
we introduce the following

\begin{definition}\label{rigid eigenspace}
An eigenspace of $P_g:\Gamma(E_g) \to \Gamma(E_g)$ is said to be a \emph{rigid eigenspace} if   it has dimension greater or equal than two, and for any two 
eigensections $u,v$  with $\| u \|_g=\|v \|_g=1$ then 
$$\|u(x)\|_x=\|v(x)\|_x \quad \quad \forall x \in M.$$

\noindent \emph{Remark.} By the polarization identity this condition is equivalent to the existence
of a function $c_g$ on $M$ so that for all $u,v$ in the eigenspace
$$(u(x), v(x))_x= c_g(x) \langle u, v \rangle_g \quad \quad \forall x \in M.$$
\end{definition}
\smallskip
In this setting, we establish the following

\begin{theorem}\label{main thm bundles}
If $P_g:\Gamma(E_g) \to \Gamma(E_g)$ has \emph{no} rigid eigenspaces, 
the set of functions $f\in C^\infty(M, \mathbb R)$ for which all the non-zero eigenvalues of $P_{e^fg}$
are simple is a residual set in $C^\infty(M, \mathbb R)$.
\end{theorem}

As a consequence of  Theorem \ref{main thm bundles}  (or Theorem \ref{main thm}) we prove

\begin{corollary}\label{main cor}
Suppose $P_g:\Gamma(E_g) \to \Gamma(E_g)$ has no rigid eigenspaces for a dense set of metrics in $\mathcal M$, 
or that it acts on $C^\infty(M)$.
Then, the set of metrics $g\in \mathcal M$ for which all non-zero eigenvalues of $P_g$ are simple is a  residual subset of $\mathcal M$.
\end{corollary}

Of course, one would like to get rid of the ``non rigidity'' assumption. Probably, this assumption
cannot be dropped if we restrict ourselves to work with conformal deformations only.
However, we believe that a generic set of deformations should break the rigidity condition.
We thereby make the following

\begin{conjecture}
 $P_g$ has no rigid eigenspaces for a  dense set of metrics in $\mathcal M$.
\end{conjecture}

\medskip
If we remove the ``non rigidity'' condition and ask the operator to satisfy the unique continuation
principle we obtain 

\begin{theorem}\label{main thm rank}
If $P_g:\Gamma(E_g) \to \Gamma(E_g)$ satisfies the unique continuation principle, 
the set of functions $f\in C^\infty(M, \mathbb R)$ for which all the non-zero eigenvalues of $P_{e^fg}$ have
multiplicity smaller than  $rank (E_g)$ is a residual set in $C^\infty(M, \mathbb R)$.
\end{theorem}

We note that for line bundles the unique continuation principle gives simplicity of eigenvalues,
for a generic set of conformal deformations.

\medskip
\begin{corollary}\label{main cor rank}
Suppose $P_g:\Gamma(E_g) \to \Gamma(E_g)$ satisfies the unique continuation principle for a dense set of metrics in $\mathcal M$.
 Then, the
set of metrics $g\in \mathcal M$ for which all non-zero eigenvalues of $P_g$ have
multiplicity smaller than the rank of the bundle is a residual subset of $\mathcal M$.
\end{corollary}

\medskip
For $c\in \mathbb R$, consider the set
$\mathcal M_{c}:= \left\{ g\in \mathcal M:\; c \notin \Spec(P_g)\right\}.$
 For $g \in \mathcal M_c$, let
$$ \mu_1(g) \leq \mu_2(g) \leq \mu_3(g) \leq  \dots $$
 be all the eigenvalues of $P_g$ in $(c, +\infty)$ counted with multiplicity. Note that it may happen that there are only finitely many
 $\mu_{i}(g)$'s. We prove

\begin{theorem}\label{continuity}
The set $\mathcal M_c$ is open, and the maps 
$$\mu_i: \mathcal M_c \to \mathbb R \qquad \quad g \mapsto  \mu_i(g) $$ 
are continuous in the $C^\infty$-topology of metrics.
 \end{theorem}

\medskip
If $P_g$ is strongly elliptic, its  spectrum is bounded below. It can be shown \cite[(7.14) Appendix]{Kod} that for a fixed metric $g_0$ there exists $c\in \mathbb R$ and a neighborhood $\mathcal V$ of $g_0$ so that $\Spec(P_g) \subset (c, +\infty)$  for all $g \in \mathcal V$ . An immediate consequence is
\begin{corollary}
If $P_g:\Gamma(E_g) \to \Gamma(E_g)$ is strongly elliptic, then its eigenvalues are continuous for $g \in \mathcal M$
in the $C^\infty$-topology.
\end{corollary} 
\medskip
\newpage
Two important  remarks:

\begin{itemize}
\item  If $P: \Gamma(E) \to \Gamma(E)$ is an elliptic, formally self-adjoint
operator acting on a smooth bundle over  compact manifold,
 its eigenvalues are real and discrete. In addition, there is
an orthonormal basis of $\Gamma_{L^2}(E)$ of eigensections of $P$.\\
\item  All the results stated above hold for \emph{non-zero} eigenvalues. Given a non-zero eigenvalue of multiplicity
greater than $1$, we use conformal transformations of the reference metric to reduce its multiplicity. This cannot be done
for zero eigenvalues for their multiplicity, $\dim \ker(P_g)$, is a conformal invariant.
\end{itemize}

\medskip
The rest of the paper is organized as follows.
In Section \ref{CCOp} we define conformally covariant operators and provide examples of operators to which our results can be applied.
In Section \ref{perturbation} we introduce the tools of perturbation theory that we shall need to split non-zero eigenvalues when they are not simple. In Section \ref{pertubed ccos} we adapt the results in perturbation theory to our class of operators and find necessary conditions to split the non-zero eigenvalues. In Section \ref{theorems} we prove Theorems \ref{main thm}, \ref{main thm bundles} and \ref{main thm rank}.
In Section \ref{continuity sec} we prove Corollary \ref{main cor}, Corollary  \ref{main cor rank},  and  Theorem \ref{continuity}.\\

 \section{Conformally Covariant Operators: definition and examples}\label{CCOp}

Next we provide examples of well known operators to which our results can be applied.
Let $g$ be a Riemannian metric over $M$ and $P_g:C^\infty(M) \to C^\infty(M)$
a (metric dependent) differential operator of order $m$. 

We say that
$P_g$ is \emph{conformally covariant of bidegree $(a,b)$} if for any conformal perturbation
of the reference metric, $g \to e^{f}g$ with $f \in C^\infty(M, \mathbb R)$, the operators $P_{e^{f}g}$ and $P_g$
are related by the formula 
$$ P_{e^fg}  =  e^{-\frac{b f}{2}} \circ P_g \circ e^{\frac{a f}{2}}.$$

If we want to consider operators acting on vector bundles the definition becomes more
involved. Let $M$ be a compact manifold  (possibly with orientation and spin structure), and $E_g$ 
a vector bundle over $M$ equipped with a bundle metric. 

\begin{definition}\label{cco}
Let $a,b \in \mathbb R$.
A \emph{conformally covariant operator $P$  of order $m$ and bidegree $(a,b)$} is a map that to every Riemannian metric $g$ over $M$ associates a differential operator $P_g:\Gamma(E_g) \to \Gamma(E_g)$ of order $m$, in such a way that 

\begin{enumerate}
\item[A)] For any two conformally related metrics,
$g$ and $e^f g$  with $f \in C^\infty(M,\mathbb R)$, there exists a bundle isomorphism $$\kappa:E_{e^f g} \to E_g$$ that preserves length fiberwise and for which
\begin{equation}\label{kappa}
 P_{e^fg}  =   \kappa^{-1} \circ e^{-\frac{b f}{2}} \circ P_g \circ e^{\frac{a f}{2}} \circ \kappa,
 \end{equation}

\item [B)] The coefficients of $P_g$ depend continuously on $g$ in the $C^\infty$-topology of metrics (see Definition \ref{cont coef}).
\end{enumerate}

For a more general formulation see \cite[pag. 4]{AJ}.
It is well known that for all these operators one always has $a\neq b$.\\
\end{definition}

\begin{definition}\label{cont coef}
The coefficients of a differential operator $P_g: \Gamma(E_g) \to \Gamma(E_g)$ are said to
\emph{depend continuously} on $g$ in the $C^k$-topology of metrics if the following is satisfied:
every metric $g_0$ has a neighborhood $\mathcal W$ in the $C^k$-topology
of metrics so that for all $ g \in \mathcal W$ there is an isomorphism 
of vector bundles $\tau_{ g}:E_{g} \to E_{ g_0}$ with the property that the coefficients
of the differential operator 
$$ \tau_{ g} \circ P_{ g} \circ \tau_{ g}^{-1}: \Gamma(E_{g_0}) \to \Gamma(E_{g_0})$$
depend continuously on $g$.
\end{definition}

We proceed to introduce some examples of operators to which our results can be applied; see \cite[pag 5]{AJ},  \cite[pag 253]{BO4},  and  \cite[pag 285]{WB} for more.\\

{\bf Conformal Laplacian.}
 On surfaces, the most common example is the Laplace operator $\Delta_g$ having
 bidegree $(0,2)$. In higher dimensions
 its generalization is the second order,  elliptic operator, named
  Conformal Laplacian, $P_{1,g}=\Delta_g+\frac{n-2}{4(n-1)}R_g$ acting on $C^\infty(M)$.
 Here  $\Delta_g=\delta_g d$ and $R_g$ is the scalar curvature.
 $P_{1,g}$ is a conformally covariant operator of bidegree $\left(\frac{n-2}{2},\frac{n+2}{2}\right)$.
 \\

{\bf Paneitz Operator.}
On compact $4$ dimensional manifolds, Paneitz discovered the $4$th order,  elliptic
operator $P_{2,g}= \Delta_g^2 + \delta_g(\frac{2}{3}R_g\,g-2Ric_g)d$ \, acting on $C^\infty(M)$. 
Here $Ric_g$ is the Ricci tensor of the metric $g$
 and both $Ric_g$ and $g$ are acting as $(1,1)$ tensors on $1$-forms.
 $P_{2,g}$ is a formally self-adjoint, conformally covariant operator of bidegree $\left(\frac{n-4}{2},\frac{n+4}{2}\right)$.
 See \cite{Pan}.
\\

{\bf GJMS Operators.}
In general, on compact manifolds of  dimension $n$ even,  if $m$ is a positive integer
such that  $2m \leq n$,
 Graham-Jenne-Mason-Sparling constructed formally self-adjoint, elliptic, conformally covariant operators $P_{m,g}$,
 acting on $C^\infty(M)$  with leading order term given by $\Delta^{m}$.
 $P_{m,g}$ is a conformally covariant operator of order $2m$ and bidegree $\left(\frac{n-2m}{2},\frac{n+2m}{2}\right)$
 that generalizes the Conformal Laplacian and the Paneitz operator to higher even orders. See \cite{GJMS}.
\\

{\bf Dirac Operator.}
Let $(M,g)$ be a compact Riemannian spin manifold. Denote its spinor bundle by $E_g$
and write $\gamma$ for the fundamental tensor-spinor. Let $\nabla$ be the connection
defined as the natural extension of the Levi-Civita connection on $TM$ to tensor-spinors
of arbitrary type.  The Dirac Operator $ \not\!\nabla_g$ is, up to normalization, the operator on $\Gamma(E_g)$
defined by $ \not\!\nabla_g=\gamma^\alpha \nabla_\alpha $.
The Dirac operator is formally self-adjoint, conformally covariant, elliptic operator of order $1$ and bidegree $\left(\frac{n-1}{2},\frac{n+1}{2}\right)$. See \cite[pag. 9]{Gin} or \cite{Hi}.
\\

{\bf Rarita-Schwinger Operator.}
In the setting of the previous example, let $T_g$ denote the twistor bundle.
The Rarita-Schwinger operator $\mathcal S^0_g$ acting on $\Gamma(T_g)$ is defined by 
$u \to \gamma^{\beta} \nabla_\beta u_\alpha - \frac{2}{n} \gamma_\alpha \nabla^\beta u_\beta$, 
where $n$ is the dimension of $M$.
$\mathcal S^0_g$ is an order $1$, elliptic,  formally self-adjoint,  conformally covariant operator
 of bidegree $\left(\frac{n-1}{2},\frac{n+1}{2}\right)$. See \cite{BH}.
 \\

{\bf Conformally Covariant Operators on forms.}
In $1982$ Branson introduced a general second order conformally covariant operator ${D_{(2,k),g}}$
on differential forms of arbitrary order $k$ and leading order term $(n-2k+2)\delta_g d + (n-2k-2)d\delta_g$
for $n \neq 1,2$ being the dimension of the manifold. Later he generalized it
to a four order operator ${D_{(4,k)},g}$ with leading order term  $(n-2k+4)(\delta_g d)^2 + (n-2k-4)(d\delta_g)^2$
for $n \neq 1,2,4$.
Both ${D_{(2,k)},g}$ and ${D_{(4,k)},g}$ are formally self-adjoint, conformally covariant operators and their leading
symbols are positive provided $k<\frac{n-2}{2}$ and $k<\frac{n-4}{2}$ respectively.
On functions, ${D_{(2,0)},g}=P_{1,g}$ and ${D_{(4,k)},g}=P_{2,g}$. See \cite[pag 276]{BO3}, \cite[pag 253]{BO4}.
For recent results and higher order generalizations see \cite{BG} and \cite{Gov04}.


\section{Background on perturbation theory}\label{perturbation}

In this section we introduce the definitions and tools we need to prove our main results.
We follow the presentation in Rellich's book \cite{Rel}, and a proof for every
result stated can be found there.\\

 Let $\Hil$ be a Hilbert space with inner product $\langle\,,\,\rangle$ and $\U$ a dense subspace of $\Hil$.
A linear operator $A$ on $\U$ is said to be \emph{formally self-adjoint},
 if it satisfies  $\langle\, Au ,  v\, \rangle$= $\langle\, u , A v\, \rangle$ for all $u,v \in \U$.
 A formally self-adjoint operator  $A$ is said to be \emph{essentially self-adjoint} if the images 
 of $A+i$ and $A-i$ are dense in $\Hil$; if these images are all of $\Hil$ we say that $A$ is \emph{self-adjoint}.\\

If $A$ is a linear operator on $\U$, its \emph{closure} is the operator
$\bar A$ defined on $\overline{\U}$ as follows: 
$\overline{\U}$ is the set of elements $u \in \Hil$ for which there exists a sequence
 $\{u_n\} \subset \U$ with $\lim_n u_n = u$ and $Au_n$ converges.
Then $\bar A u:= \lim_n Au_n$. We note that if $A$ is formally self-adjoint, so is $\bar A$.\\

A family of linear operators $A(\ep)$ on $\U$ indexed by $\ep \in \mathbb R$
 is said to be \emph{regular} in a neighborhood of $\ep=0$ if 
 there exists a bounded bijective operator $U:\Hil \to \U$ so that for all $v \in \Hil$,
 $A(\ep)[ U (v)]$ is a regular element, in the sense that it is a power series 
 convergent in a neighborhood of $\ep=0$.
Finding the operator $U$ is usually very difficult. Under certain conditions on the operators,
proving regularity turns out to be much simpler. To this end, we introduce the
following criterion.

\begin{criterion}\label{crit3 rellich}(\cite[page 78]{Rel})
Suppose that $A(\ep)$ on $\U$ is a family of formally self-adjoint
operators in a neighborhood of $\ep=0$. Suppose that $A^{(0)}=A(0)$ is
essentially self-adjoint, and there exist formally self-adjoint operators
$A^{(1)}, A^{(2)}, \dots$ on $\U$ such that for all $u \in \U$
$$A(\ep)u=A^{(0)}u+\ep A^{(1)}u + \ep^2 A^{(2)}u + \dots$$
Assume in addition that there exists $a\geq0$ so that
$$\|A^{(k)}u\| \leq a^k \|A^{(0)}u\|, \quad \text{for all } k=1,2,\dots$$
Then, on $\U$, $A(\ep)$  is essentially self-adjoint and 
$\overline{A(\ep)}$ on $\overline{\U}$ is regular in a neighborhood of $\ep=0$.
\end{criterion}

For the purpose of splitting non-zero eigenvalues, next proposition plays a key role.

\begin{proposition}\label{rellich}(\cite[page 74]{Rel})
	Suppose that $B(\ep)$ on $ \U$ is a family of regular, formally self-adjoint
	operators in a neighborhood of $\ep=0$. Suppose that $B^{(0)}=B(0)$ is  self-adjoint.
	 Suppose that $\lambda$ is an eigenvalue of finite multiplicity
	$\ell$ of the operator $B(0)$, and suppose there are positive numbers $d_1,d_2$ such that the
	spectrum of $B(0)$ in $(\lambda-d_1,\lambda+d_2)$ consists exactly of the eigenvalue  $\lambda$.
	
	Then, there exist power series $\lambda_1(\ep), \dots, \lambda_\ell(\ep)$ convergent 
	in a neighborhood of $\ep=0$
	and power series $u_1(\ep), \dots, u_\ell(\ep),$ satisfying

		\begin{enumerate}
		\item  $u_i(\ep)$ converges for small $\ep$ in the sense that the partial sums converge in 
		$\Hil$ to an element in $\U$, for $i=1 \dots \ell$.
		\item $B(\ep)u_i(\ep)=\lambda_i(\ep)u_i(\ep)$ and $\lambda_i(0)=\lambda$, for $i=1, \dots, \ell$.
		\item $\langle u_i(\ep), u_j(\ep) \rangle= \delta_{ij}$, for $i,j=1, \dots, \ell$.
		\item For each pair of positive numbers $d_1', d_2'$ with $d_1'<d_1$ and $d_2'<d_2$,
		there exists a positive number $\delta$ such that the spectrum of $B(\ep)$ in
		$[\lambda-d'_1,\lambda+d'_2]$ consists exactly of the points $\lambda_1(\ep), \dots, \lambda_\ell(\ep),$
		for $|\ep|<\delta$.
		\end{enumerate}
\end{proposition}

We note that since  $B(\ep)u_i(\ep)=\lambda_i(\ep) u_i(\ep)$, differentiating with respect to $\ep$ both sides of
the equality we obtain
$$\langle B^{(1)}(\ep)u_i(\ep), u_j(\ep) \rangle + \langle u'_i(\ep), B(\ep)u_j(\ep) \rangle 
=\langle \lambda_i'(\ep)u_i(\ep), u_j(\ep) \rangle + \langle u'_i(\ep), \lambda_i(\ep)u_j(\ep) \rangle.$$

When $i=j$ the above equality translates to
\begin{equation}\label{lambda'}
\lambda_i'(\ep)=\langle B^{(1)}(\ep) u_i(\ep), u_i(\ep) \rangle.
\end{equation}

Also, evaluating at $\ep=0$ we get
\begin{equation}\label{lambda'(0)}
\lambda_i'(0)=\langle B^{(1)} u_i(0), u_j(0) \rangle \delta_{ij}.
\end{equation}

\section{Perturbation theory for Conformally Covariant operators}\label{pertubed ccos}

Consider a conformal change of the reference metric $ g \to e^{\ep f}g$  for $f\in C^\infty(M)$
 and $\ep \in \mathbb R$. Since $P_g:\Gamma(E_g)\to \Gamma(E_g)$ is a 
 \emph{conformally covariant operator} of bidegree $(a,b)$,
there exists $\kappa:E_{e^{\ep f}g}\to E_g$, a bundle isomorphism that preserves the length fiberwise, so that
$$ P_{e^{\ep f}g}  = \kappa^{-1} \circ  e^{-\frac{b\ep f}{2}} \circ P_g \circ e^{\frac{a \ep f}{2}} \circ \kappa.$$

We work with a modified version of $P_{e^{\ep f}g}$. For $c:=\frac{a+b}{4}$ set
$$\eta:=c-\frac{b}{2}=\frac{a}{2}-c$$ and define
$$A_f(\ep): \Gamma(E_g) \to \Gamma(E_g),\quad \quad A_f(\ep):=e^{\eta\ep f} \circ P_g \circ e^{\eta \ep f}.$$

The advantage of working with these operators is that, unlike $P_{e^{\ep f}g}$, they are formally self-adjoint with respect to $\langle \,,\, \rangle_g$.  Note that $\eta \neq 0$ for $a \neq b$, and observe that
\begin{align*}
A_f(\ep) &= e^{\eta\ep f} \circ P_g \circ e^{\eta \ep f}\\
	&=e^ {c \ep f} e^{-\frac {b \ep f}{2}} \circ P_g \circ e^{\frac {a \ep f}{2}} e^ {-c \ep f} \\
	&= \kappa \circ e^{c \ep f}\circ P_{e^{\ep f}g} \circ e^{- c \ep f} \circ \kappa^{-1}.
\end{align*}

\begin{remark}\label{same evs}
$A_f(\ep)$ and $P_{e^{\ep f} g}$ have the same eigenvalues. Indeed,
 $u(\ep)$ is an eigensection of $P_{e^{\ep f}g}$ with eigenvalue
$\lambda(\ep)$ if and only if $\kappa (e^{c\ep f}u(\ep))$ is an eigensection
for $A_f(\ep)$ with the same eigenvalue. 
\end{remark}

$A_f(\ep)$ is a deformation of $P_g=A_f(0)$ that has the same spectrum
as $P_{e^{\ep f}g}$ and is formally self-adjoint with respect to $\langle , \rangle_g$.
Note also that $A_f(\ep)$ is elliptic so there exists a basis of $\Gamma_{L^2}(E_{g})$
of eigensections of $A_f(\ep)$.

\begin{lemma}\label{A (k), 2}
With the notation of  Criterion \ref{crit3 rellich}, the operators $A_f^{(k)}(\varepsilon):=\frac{1}{k!}\frac{d^k}{d\ep^k} A_f(\ep)$
 are formally self-adjoint and 
\begin{equation}\label{ineq A'}
\left \|A_f^{(k)}(\ep) u \right \|_g \leq \frac{\left(2\, |\eta|\, \|f\|_\infty\right)^k}{k!}\;  \|A_f(\ep) u\|_g
\quad \quad \forall u \in \Gamma(E_g).
\end{equation}
\end{lemma}

\begin{proof}
Since $A_f(\ep)$ is formally self-adjoint, so is $A_f^{(k)}(\ep)$. Indeed, for $u,v \in \Gamma(E_g)$, 
$\langle A_f(\ep)u,v \rangle_g-\langle u, A_f(\ep)v \rangle_g=0$. Hence,
$0=\frac{d^k}{d\ep^k}(\langle A_f(\ep)u,v \rangle_g-\langle u, A_f(\ep)v \rangle_g )\big|_{\ep=0}=$\\
$=k!(\langle A_f^{(k)}u, v \rangle_g- \langle u, A_f^{(k)}v \rangle_g)$.
For the norm bound, observe that
\begin{equation}\label{derivatives of A}
\frac{d^k}{d\ep^k}\left[A_f(\ep)(u)\Big.\right]= \eta^k \;\sum_{l=0}^k {k \choose l} f^{k-l} 
 A_f(\ep) (f^l u),
\end{equation}
and notice that from the fact that $A_f(\ep)$ is 
formally self-adjoint it also follows that 
$\|A_f(\ep)(hu)\|_g\leq \|h\|_\infty \|A_f(\ep)u\|_g$ for all $h \in C^\infty(M)$.
\end{proof}

In the following Proposition we show how to split the multiple eigenvalues of $P_g$.
From now on we write $A_f^{(k)}:=A_f^{(k)}(0)$.

\begin{proposition}\label{splitting ev}
Suppose $\lambda$ is a non-zero eigenvalue of $P_g$.
Let $V_\lambda$ be the eigenspace of eigenvalue $\lambda$  and $\Pi$ the orthogonal projection onto it.
With the notation of Proposition \ref{rellich},  if $\Pi \circ A_f^{(1)}|_{V_\lambda}$ is not a multiple of the identity, there
exists $\ep_0>0$ and a pair $(i,j)$ for which $\lambda_i(\ep)\neq \lambda_j(\ep)$ for all $0<\ep<\ep_0.$
\end{proposition}

\begin{proof}
Assume the results of Proposition \ref{rellich} are true for $B(\ep)=\overline{A_f(\ep)}$, and note that
for there is a basis of $\Gamma_{L^2}(E_g)$ of eigensections of  $A_f(\ep)$, the eigensections of $\overline{A_f(\ep)}$ and $A_f(\ep)$ coincide. 
By relation \eqref{lambda'(0)},  $\lambda'_1(0),\dots,\lambda'_\ell(0)$ are the eigenvalues of \,$\Pi \circ A_f^{(1)} |_{V_\lambda}$. Since $\Pi \circ A_f^{(1)} |_{V_\lambda}$ is not a multiple
of the identity, there exist $i, j$ with $\lambda_i'(0) \neq \lambda_j'(0)$ and this implies
that $\lambda_i(\ep) \neq \lambda_j(\ep)$ for small $\ep$, which by Remark \ref{same evs}
is the desired result.
We therefore proceed to show that all the assumptions in Proposition \ref{rellich} are
satisfied for $B(\ep)=\overline{A_f(\ep)}$, $U=\overline{\Gamma(E_g)}$ and $\Hil=\Gamma_{L^2}(E_g)$.\\

\emph{ $\overline{A_f(0)}=\overline{P_g}$ is self-adjoint:} This follows from the fact that
$P_g$ is essentially self-adjoint, and the closure of an essentially self-adjoint is a self-adjoint operator.
To see that $A_f(0)=P_g$ is essentially self-adjoint, note that since
there is a basis of $\Gamma_{L^2}(E_g)$ of eigensections of $P_g$, it is enough to show that for any eigensection $\phi$ of eigenvalue 
 $\lambda$ there exist $u,v \in \Gamma(E_g)$ for which $P_g u + i u =\phi$ and $P_g v - i v=\phi$. 
Thereby, it suffices
to set $u=\frac{1}{\lambda +i} \phi$ and $v=\frac{1}{\lambda-i} \phi$.\\

\emph{$\overline{A_f(\ep)}$ is regular on $\Gamma(E_g)$:}
From Lemma \ref{A (k), 2} and Criterion \ref{crit3 rellich} applied to $A(\ep)=A_f(\ep)$,
we obtain that $A_f(\ep)$ is a family of operators on 
$\Gamma(E_g)$ which are essentially self-adjoint and their closure  $\overline{A_f(\ep)}$ 
on $\overline{\Gamma(E_g)}$ is regular.

\end{proof}


\subsection{Splitting eigenvalues}
Recall from Definition \ref{rigid eigenspace}  that an eigenspace of $P_g$ is said to be a \emph{rigid eigenspace} 
if  it has dimension greater or equal than two, and for any two eigensections $u,v$  with $\| u \|_g=\|v \|_g=1$ one has
$$\|u(x)\|_x=\|v(x)\|_x \quad \quad \forall x \in M.$$

Being an operator with \emph{no} rigid eigenspaces is the condition that will allow us to split eigenvalues.
For this reason, at the end of this section we show that operators acting on
$C^\infty(M)$ have no rigid eigenspaces (see Proposition \ref{smooth functions, no rigidity}).

Our main tool is the following

\begin{proposition}\label{prop splitting}
Suppose  $P_g$ has no rigid eigenspaces.
Let $\lambda$ be a non-zero eigenvalue of $P_g$ of multiplicity $\ell \geq 2$. 
 Then, there exists a function
 $f \in C^\infty(M, \mathbb R)$ and $\ep_0>0$ so that among the perturbed eigenvalues 
$\lambda_1(\ep),\dots, \lambda_\ell (\ep)$ of $P_{e^{\ep f}g}$  there 
exists a pair $(i,j)$ for which $\lambda_i(\ep)\neq \lambda_j(\ep)$ for all $0<\ep<\ep_0.$
\end{proposition}

\begin{proof}

Since $P_g$ has no rigid eigenspaces, there exist $u,v \in \Gamma(E_g)$ 
linearly independent normalized eigensections in the $\lambda$-eigenspace so that 
$\|u(x)\|^2_x \neq \|v(x)\|^2_x$ for some $x\in M$. For such sections there exists
$f\in C^\infty(M, \mathbb R)$ so that $ \langle fu,u \rangle_g \neq  \langle fv,v \rangle_g$.
To prove our result, by Proposition \ref{splitting ev} it would suffice to show that 
$$\langle A_f^{(1)} u,u \rangle_g \neq \langle A_f^{(1)} v,v \rangle_g.$$

Using that $P_g$ is formally self-adjoint and evaluating equation \eqref{derivatives of A} at $\ep=0$ 
(for $k$=1) we have 
$$\langle A_f^{(1)} u,u \rangle_g=\eta \, \langle fP_g(u)+P_g(fu),u \rangle_g
=2 \eta \,\lambda \langle fu,u \rangle_g,$$ 
and similarly, $\langle A_f^{(1)} v,v \rangle_g=2 \eta \,\lambda \langle fv,v \rangle_g$.
The result follows.
\end{proof}

A weaker but more general result is the following

\begin{proposition}\label{prop splitting rank}
Suppose $P_g:\Gamma(E_g) \to \Gamma(E_g)$ satisfies the unique continuation principle.
Let $\lambda$ be a non-zero eigenvalue of $P_g$ of multiplicity $\ell > rank(E_g)$.
Then, there exists $\ep_0>0$ and a function
 $f \in C^\infty(M, \mathbb R)$ so that among the perturbed eigenvalues 
$\lambda_1(\ep),\dots, \lambda_\ell (\ep)$ of $P_{e^{\ep f}g}$  there is a pair $(i,j)$
for which $\lambda_i(\ep)\neq \lambda_j(\ep)$ for all $0<\ep<\ep_0$.

\end{proposition}\label{bad bundle}
\begin{proof}
Let $\{u_1,\dots,u_\ell\}$ be an orthonormal basis of the $\lambda$-eigenspace. 
If for some $i \neq j$ there exists $x \in M$ for which $\|u_i(x)\|_x\neq \|u_j(x)\|_x$ we proceed as in 
Proposition \ref{prop splitting} and find $f \in C^\infty(M,\mathbb R)$ for which $\langle fu_i,u_i \rangle_g \neq 
\langle f u_j , u_j \rangle_g$. We show that under our assumptions this is the only possible situation.

If for any two normalized eigensections $u,v \in \Gamma(E_g)$ of eigenvalue $\lambda$ we had
$\|u(x)\|^2_x=\|v(x)\|^2_x$ for all $x\in M$, then by the polarization identity (see remark in Definition \ref{rigid eigenspace})
we would obtain  $(u_i(x),u_j(x))_x=0$ for all $i\neq j$ and $x\in M$. 
  By the rank condition, for some $i=1,\dots, \ell$ the section $u_i$ has to vanish on an open set,
and by the unique continuation principle it must vanish everywhere, which is a contradiction. 
\end{proof}

We finish this section translating the previous results to the setting of smooth functions.

\begin{proposition}\label{smooth functions, no rigidity}
Operators acting on $C^\infty(M)$ have no rigid eigenspaces.
\end{proposition}

\begin{proof}
Let $ \tilde u, \tilde v$ be two linearly independent, orthonormal
eigenfunctions of $P_g$ with eigenvalue $\lambda$. 
Set $D:=\{x\in M:\, \tilde u(x)\neq \tilde v(x)\}$. 
If there is $x\in D$ with $\tilde u(x)\neq - \tilde v(x)$, the functions $u=\tilde u$ and $v=\tilde v$ 
break the rigidity condition.
If for all $x\in D$ we have $\tilde u(x)=-\tilde v(x)$, the functions $u=\frac{\tilde u+\tilde v}{\| \tilde u+\tilde v\|_g}$ 
and  $v=\frac{\tilde u-\tilde v}{\|\tilde u-\tilde v\|_g}$ do the job. 
Indeed, $v=0$ on $D^c$ and there exists $x\in D^c$ for which $ u(x)\neq 0$ because otherwise
$\tilde u\equiv - \tilde v$  and this contradicts the independence.
\end{proof}


\section{Eigenvalue multiplicity for conformal deformations}\label{theorems}
In this section we address the proofs of Theorems \ref{main thm}, \ref{main thm bundles} and \ref{main thm rank}.

\subsection{Proof of Theorems \ref{main thm} and \ref{main thm bundles}.} 
Given $\alpha \in \mathbb N$, and $g\in \mathcal M$ consider the set
\begin{align*}
F_{g,\alpha}:=\Big \{f\in C^\infty(M, \mathbb R): \;\, \lambda \text{ is \emph{simple} for all }\,
\lambda \in  \Spec(P_{e^f g})\cap \big([-\alpha,0)\cup(0, \alpha]\big)\Big \}.
\end{align*}

The set of functions $f\in C^\infty(M, \mathbb R)$ for which all the non-zero eigenvalues of $P_{e^fg}$
are simple coincides with the set $\bigcap_{\alpha \in \mathbb N} F_{g,\alpha}.$
To show that the latter is a residual subset of $C^\infty(M, \mathbb R)$, 
we prove that the sets $F_{g,\alpha}$ are open and dense in $C^\infty(M, \mathbb R)$.\\

We note that for conformal metric deformations, the multiplicity of the zero eigenvalue remains fixed. Indeed, 
according to \eqref{kappa}, for $u \in \Gamma(E_g)$ and $f \in C^\infty(M, \mathbb R)$,
we know 
$$P_g (u)= 0 \quad\quad \text{if and only if }\quad\quad P_{e^fg} (\kappa^{-1} (e^{-\frac{af}{2}}u))=0.$$
\underline{Throughout this subsection we assume the hypothesis of Theorems 
\ref{main thm} or \ref{main thm bundles} hold.}\\

\subsubsection{$F_{g,\alpha}$ is dense in $C^\infty(M, \mathbb R)$}\label{dense} \ \\

Fix $f_0 \notin F_{g,\alpha}$ and let $W$ be an open neighborhood of $f_0$. Since  at
least one of the  eigenvalues in $[-\alpha, 0) \cup (0, \alpha]$ has multiplicity greater than two, we  proceed to split it.
By Proposition  \ref{prop splitting} (and Proposition \ref{smooth functions, no rigidity} when the
operator acts on $C^\infty(M)$)
 there exists $f_1\in C^\infty(M, \mathbb R)$ for which at least two of the first $\alpha$ non-zero
eigenvalues of $P_{e^{\ep_1f_1} (e^{f_0}g)}$ are different as long
as $\ep_1$ is small enough. Moreover, those eigenvalues that were simple would remain being simple
for such $\ep_1$. Also, for $\ep_1$ small enough, we can assume that none of the eigenvalues that
originally belonged to $[-\alpha, \alpha]^c$ will have perturbations belonging to $[-\alpha, \alpha]$. Let $\ep_1$ be small as before and so that $\ep_1 f_1 + f_0 $ belongs to $W$.
If $\ep_1 f_1 + f_0 $ belongs to $F_{g,\alpha}$ as well, we are done. If not, in finitely many steps, the repetition
of this construction will lead us to a function $\ep_Nf_N+\dots+\ep_1f_1 +f_0$ in $W\cap F_\alpha$. 
Hence, $F_{g,\alpha}$ is dense.

\subsubsection{$F_{g,\alpha}$ is open in $C^\infty(M, \mathbb R)$}\label{open}\ \\

Fix $f_0\in F_{g,\alpha}$.
In order to show that $F_{g,\alpha}$ is open we need to establish how rapidly the eigenvalues
of $A_f(\ep)$ grow with $\ep$. From now on we restrict ourselves to perturbations
of the form $ e^{\ep f}(e^{f_0}g)$ for $f\in C^\infty(M, \mathbb R)$ with $\|f\|_{\infty}<1.$ 
Let $u(\ep)$ be an eigensection of $A_f(\ep)$ with eigenvalue $\lambda(\ep)$. Equation \eqref{lambda'} gives
$\left|\lambda '(\ep)\right|\leq \|A_f^{(1)}(\ep) u(\ep)\|_g$ for $j=1,\dots, \alpha$. 
Putting this together with inequality \eqref{ineq A'} for $k=1$ we get

$$\left|\lambda'(\ep) \right| \leq 2|\eta|\,  \|A_f(\ep) u(\ep)\|_g
= 2|\eta|\, | \lambda(\ep)|.$$

The solution of the differential inequality leads to the following bound for the growth of the perturbed eigenvalues:

$$\left| \lambda(\ep)- \lambda \right| \leq  |\lambda| \left(e^{2|\eta|\,  |\ep|}-1 \right), \quad |\ep|<\delta.$$

 Write $\lambda_1 \leq \lambda_2 \leq \dots \leq \lambda_\kappa$ for all the eigenvalues (repeated according to multiplicity) of $P_{e^{f_0}g}$ that belong to $[-\alpha, 0) \cup (0, \alpha]$.
Let $d_1, \dots, d_\kappa$ be so that the intervals $[\lambda_j-d_j,  \lambda_j +d_j]$ for
$j=1,\dots, \kappa$, are disjoint. Write $\lambda_{\beta^-}$ for the biggest eigenvalue in $(-\infty, -\alpha)$ and
 $\lambda_{\beta^+}$ for the smallest eigenvalue in $(\alpha, +\infty)$. We further assume that 
$\lambda_{\beta^-} \notin [ \lambda_{1} -d_1,  \lambda_{1} +d_1]$ and $\lambda_{\beta^+} \notin [ \lambda_{\kappa} -d_\kappa,  \lambda_{\kappa} +d_\kappa]$.
\\ 
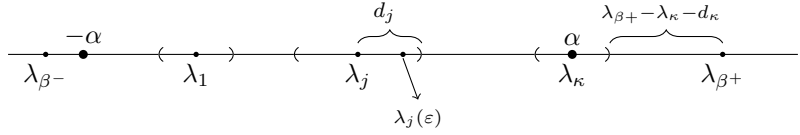
\begin{center}
  \begin{tikzpicture}
          \draw[-] (-1,0) -- (1,0);
          \draw[(-)] (1,0) -- (2,0);
          \draw[-] (2,0) -- (2.8,0);
          \draw[(-)] (2.8,0) -- (4.5,0);
          \draw[-] (4.5,0) -- (6,0);
          \draw[(-)] (6,0) -- (7,0);
          \draw[-] (7,0) -- (9.5,0);

	 \node[below] at (-0.5, 0) {$\lambda_{\beta^-}$};
	    \node[above] at (0, 0) {$-\alpha$};
          \node[below] at (1.5, 0) {$\lambda_1$};
          \node[below] at (3.65, 0) {$\lambda_j$};
          \node[below] at (6.5, 0) {$\lambda_{\kappa}$};
           \node[above] at (6.5, 0) {$\alpha$};
          \node[below] at (8.5, 0) {$\lambda_{\beta^+}$};
          
           \fill (-0.5,0) circle (1pt) ;
          \fill (0,0) circle (1.7pt) ;
          \fill (1.5,0) circle (1pt) ;
          \fill (3.65,0) circle (1pt);
          \fill (4.25,0) circle (1pt);
          \fill (6.5,0) circle (1.7pt);
          \fill (8.5,0) circle (1pt);
          
          \draw [black,decorate,decoration={brace,amplitude=5pt}, xshift=0pt,yshift=0pt] (3.65,0.15)  -- (4.5,0.15)
           node [black,midway,above=4pt,xshift=-2pt] {\footnotesize $d_j$};
          \draw [black,decorate,decoration={brace,amplitude=5pt}, xshift=0pt,yshift=0pt] (7,0.15)  -- (8.5,0.15)
          node [black,midway,above=4pt,xshift=-2pt] { $\scriptstyle{ \lambda_{\beta+}-\lambda_\kappa-d_\kappa}$};

          \draw [->] (4.25,-0.05) -- (4.45,-0.6);
          \node[below] at (4.45, -0.6) {${\scriptstyle \lambda_j(\ep)}$};
        \end{tikzpicture}
        \end{center}

In order to ensure that for each $j=1,\dots, \alpha$ the perturbed eigenvalue 
$\lambda_{j}(\ep)$ belongs to $ [\lambda_{j} -d_j,  \lambda_{j} +d_j]$, select  $0<\delta_1 \leq \delta$,
so that whenever $|\ep|<\delta_1$ we have that
$|\lambda_j(\ep)-\lambda_j|\leq |\lambda_{j}| \left(e^{2|\eta|\,  |\ep|}-1 \right) \leq d_j$
for all $j=1,\dots, \kappa$.

To finish our argument, we need to make sure that none of the perturbations of the eigenvalues  that
initially belonged to $(-\infty, -\alpha) \cup (\alpha, +\infty)$
  coincide with the perturbations 
corresponding to $\lambda_{1},\dots, \lambda_{\kappa}$. To such end, it is enough
to choose $0<\delta_2 \leq \delta $ so that for  $|\ep|<\delta_2$,
\begin{align*}
|\lambda_{\beta^+}|\left(e^{2|\eta|\,  |\ep|}-1 \right)
&<\min \{ \lambda_{\beta^+}-\lambda_\kappa-d_\kappa, \,\lambda_{\beta^+}-\alpha\},\\
\intertext {and}
|\lambda_{\beta^-}|\left(e^{2|\eta|\,  |\ep|}-1 \right)
&<\min \{ \lambda_{1}-d_1-\lambda_{\beta^-}, \,-\alpha-\lambda_{\beta^-}\}.
\end{align*}

Summing up, if $\|f\|_\infty<1$  and $|\ep|<\min\{\delta_1,\delta_2\}$, then $\ep f +f_0 \in F_{g,\alpha}$.
Or in other words, $\{f_0+f: \, \|f\|_\infty<\ep\} \subset F_{g,\alpha}$, so $F_{g,\alpha}$ is open.

\subsection{Proof of Theorem \ref{main thm rank}.}

The set of functions $f\in C^\infty(M, \mathbb R)$ for which all the eigenvalues of $P_{e^fg}$
have multiplicity smaller than $rank (E_g)$  can be written as $\cap_{\alpha \in \mathbb N} \hat F_{g,\alpha}$ where
\begin{align*}
&\hat F_{g,\alpha}:=\Big \{f\in C^\infty(M, \mathbb R): \;\;
\dim \Ker (P_{e^fg}-\lambda) \leq rank (E_g)  \\
&\hspace{5cm }\text{for all } \lambda \in  \Spec(P_{e^f g})\cap \big([-\alpha,0)\cup(0, \alpha]\big)\Big \}.
\end{align*}

 $\hat F_{g,\alpha}$ is dense in $C^\infty(M, \mathbb R)$ by the same argument presented in \ref{dense}, using Proposition 
 \ref{prop splitting rank} to find the $f_i$'s. The proof for  $\hat F_{g,\alpha}$ being open in $C^\infty(M, \mathbb R)$ is 
 analogue to that in \ref{open}.

\section{Local continuity of eigenvalues}\label{continuity sec}

The arguments we present in this section are an adaptation 
of the proof of Theorem 2 in  \cite{KS} by Kodaira and Spencer;
they prove similar results to Theorem \ref{continuity} for \underline {strongly} elliptic operators that have coefficients that depend  continuously on a
parameter \underline{$t \in \mathbb R^n$} in the $C^\infty$-topology.\\

From now on fix a Riemannian metric $g_0$. 
Let $\{X_i\}_{i\in I}$ be a finite covering of $M$ with local coordinates $(x_i^1,\dots, x_i^n)$ on each neighborhood
$X_i$ and let $u \in \Gamma(E_g)$ be represented in the form $(u_i^1(x), \dots,u_i^\mu(x))$
for $x \in X_i$  and $\mu=rank(E_g)$. 
For each integer $k$ define the $k$- norm
$$\|u\|_k^2:= \sum_{\nu=1}^\mu \sum_{\underset{|\alpha|\leq k}{\alpha}}  \sum_{i \in I}
 \int_{X_i} |\partial_i^\alpha u_i^\nu(x)|^2 \; dvol_{g_0},$$
where $\partial_i^{\alpha}:=\partial_ {x_i^1}^{\alpha_1} \dots \partial_ {x_i^n}^{\alpha_n}$  and
$|\alpha|=\alpha_1+\dots +\alpha_n$. We note that $\|\cdot \|_0=\| \cdot\|_{g_0}$ and that
the $k$-norm just introduced is equivalent to the $k$-Sobolev norm. \\

By the continuity of the coefficients of $P_g$ (see Definition \ref{cont coef}),
there exists  $\mathcal W_{g_0}$  neighborhood of $g_0$ in the $C^\infty$-topology
of metrics, so that for every metric $ g \in \mathcal W_{g_0}$ there is an isomorphism 
of vector bundles $\tau_{ g}:E_{g} \to E_{ g_0}$ with the property that the coefficients
of the differential operator 
\begin{equation}\label{Q is elliptic}
Q_g:= \tau_{ g} \circ P_{ g} \circ \tau_{ g}^{-1}: \Gamma(E_{g_0}) \to \Gamma(E_{g_0})
\end{equation}
depend continuously on $g \in \mathcal W_{g_0}$.
The following analogue of Lemma 3 in \cite{KS} holds:

\begin{lemma}\label{fried}
There exists a neighborhood  $\mathcal W \subset \mathcal W_{g_0}$ of $g_0$ 
so that   for every integer $k \geq 0$  there is a constant $c_k$ independent of
 $g \in \mathcal W$ for which
$$ \|u \|_{k+m}^2 \leq c_k \, \left(\; \|Q_g u\|_k^2 + \|u\|^2_0 \; \big.\right) , $$
for all $ u \in \Gamma_{L^2}(E_{g_0})$ and $g \in \mathcal W$.
\end{lemma}

\begin{proof}
Since $P_g$ is elliptic, from relation \eqref{Q is elliptic}
we deduce that $Q_g$ is elliptic as well. 
By Theorem $5.2$ part ($iii$)  in \cite[p.193]{LM}, for every positive integer $k$ there exists a constant $c_k$
so that for all $ u \in \Gamma(E_{g_0})$,
 $\|u \|_{k+m}^2 \leq c_k \, \left(\; \|Q_{g_0} u\|_k^2 + \|u\|^2_k\right)$. 
 By induction on $k$ and the Sobolev embedding Theorem we obtain 
$$\|u \|_{k+m}^2 \leq c_k \, \left(\; \|Q_{g_0} u\|_k^2 + \|u\|^2_0\right).$$
The result follows from the continuity of the coefficients of $Q_g$ for $g \in \mathcal W_{g_0}$.
\end{proof}

Since $P_g$ is  elliptic and formally self-adjoint, its spectrum $\Spec(P_g)$ is real and discrete.
Note that the spectrum of $P_g$ and $Q_g$ coincide. Indeed,  $u$ is an eigensection of $P_g$
with eigenvalue $\lambda$ if and only if $\tau_g  u$ is an eigensection of $Q_g$ with eigenvalue $\lambda$.
Fix $\xi \in \mathbb C$ and define $$Q_g(\xi):=Q_g -\xi.$$
It is well known that $Q_g(\xi)$ is surjective provided $\xi$ belongs to the resolvent set of $Q_g$ (i.e. $\xi \notin \Spec(P_g)$). 
Furthermore, for $\xi_0$ in the resolvent set of $P_{g_0}$, set $b_{g_0}:= \inf_{\lambda \in \Spec(P_{g_0})} |\lambda-\xi_0|$. We then know
$$\|Q_{g_0}(\xi_0) u \|_{0} \geq b_{g_0} \|u\|_{0}.$$

\begin{lemma}\label{resolvent}
There exists $\delta >0 $ and $\mathcal V \subset \mathcal W_{g_0}$ neighborhood of $g_0$ 
 so that  the resolvent operator $R_g(\xi):=Q_g(\xi)^{-1}$  exists for  $g \in \mathcal V$ and $|\xi -\xi_0|<\delta$. 
 In addition, for every $u \in \Gamma(E_{g_0})$ the section $R_g(\xi) u$ depends continuously
 on $\xi$ and $g$ in the $\|\cdot\|_0$ norm.
\end{lemma}
\begin{proof}

We first prove the injectivity of $Q_g(\xi)$. Let $\mathcal W$ be as in Lemma \ref{fried}. It suffices to show that for all $\varepsilon>0$ there exists 
$\delta >0 $ and $\mathcal V \subset \mathcal W$ so that 
$$\|Q_{g}(\xi) u \|_{0} \geq (b_{g_0}-\varepsilon) \|u\|_{0}$$
for  $g \in \mathcal V$ and $|\xi -\xi_0|<\delta$.
We proceed by contradiction. Suppose there exists $\varepsilon>0$ together with
a sequence
 $\{(\delta_i, \mathcal V_i, u_i, \xi_i)\}_i$, with $\delta_i \overset{i}{\to} 0$, $\mathcal V_i$ shrinking around $g_0$, 
and $|\xi_i-\xi_0|<\delta_i$,
 such that
$$\|Q_{g_i}(\xi_i) u_i \|_{0} < (b_{g_0}-\varepsilon) \|u_i\|_{0}$$ for $g_i \in \mathcal V_i$ and
$|\xi_i-\xi_0|<\delta_i$. Without loss of generality assume $\|u_i\|_0=1$.

By Lemma \ref{fried} we know $\|u_i\|_m \leq c_0(b_{g_0}-\varepsilon)$, and by the continuity in $g$
of the coefficients of $Q_g$, it follows that $\| (Q_{g_i}(\xi_i)- Q_{g_0}(\xi_0))u_i\|_0 \to 0$.
Since 
$$\| (Q_{g_i}(\xi_i)- Q_{g_0}(\xi_0))u_i\|_0 \geq \|Q_{g_0}(\xi_0) u_i\|_0- \|Q_{g_i}(\xi_i) u_i\|_0 \geq b_{g_0}-(b_{g_0}-\ep)= \varepsilon,$$ we obtain the desired contradiction.

To prove the continuity statement notice that 
\begin{align*}
\|R_g(\xi) u - R_{g_0}(\xi_0)u\|_0 
&\leq \frac{1}{b_{g_0}} \|Q_g(\xi)R_g(\xi) u - Q_g(\xi) R_{g_0}(\xi_0)u\|_0\\
&=\frac{1}{b_{g_0}} \|Q_{g_0}(\xi_0)R_{g_0}(\xi_0) u - Q_g(\xi) R_{g_0}(\xi_0)u\|_0\\
&=\frac{1}{b_{g_0}}\left( \|\left(Q_{g_0} -Q_g \big.\right)(R_{g_0}(\xi_0) u)\|_0
+ |\xi-\xi_0|\|(R_{g_0}(\xi_0) u)\|_0 \Big.\right),
\end{align*}
and use the continuity in $g$ of the coefficients of $Q_g$.
\end{proof}

Let $g_0 \in \mathcal M$ and continue to write $\mathcal W_{g_0}$ for the neighborhood of $g_0$
for which  the vector bundle isomorphism    $\tau_{ g}:E_{g} \to E_{ g_0}$  is defined for all $g \in \mathcal W_{g_0}$.
 Let $\mathbb C$ be a differentiable curve with interior domain $D \subset \mathbb C$.
For $g \in \mathcal W_{g_0}$, write $\mathbf F_g(C)$ for the linear subspace of $\Gamma(E_{g_0})$
\
$$\mathbf F_g(C):= span \left \{ \tau_g u:\; \; u \in \Ker (P_g - \lambda I) \text{ for }\lambda \in D \cap \Spec(P_g)  \big.\right  \}.$$

Note that 
\begin{equation}\label{dim}
 \dim \mathbf F_g(C) = \sum_{\lambda \in D \cap \Spec(P_g) } \dim \Ker (P_g - \lambda I).
 \end{equation}

\begin{proposition}\label{dimensions}
If $C$ meets none of the eigenvalues of $P_{g_0}$, then there exists a neighborhood
$\mathcal V \subset \mathcal W_{g_0}$ of $g_0$ so that for all $g \in \mathcal V$
\begin{equation}\label{dim F}
\dim \mathbf F_g(C) = \dim \mathbf F_{g_0} ( C).
\end{equation}
\end{proposition}

\begin{proof}\ \\
\emph{Step 1.} For $g \in  \mathcal W_{g_0}$, define the spectral projection operator $F_g(C): \Gamma(E_{g_0}) \to \Gamma(E_{g_0})$
 to be the projection of $\Gamma(E_{g_0})$ onto $\mathbf F_g(C)$.
Since $C$ meets none of the eigenvalues of $P_{g_0}$, by Lemma \ref{resolvent} there exist a neighborhood $C'$ of the curve $C$ and a neighborhood $\mathcal V' \subset \mathcal W$ of $g_0$
so that none of the eigenvalues of $P_g$ belong to $C'$ for $g \in \mathcal V'$.
By holomorphic functional calculus  
$$ F_g(C)\, u =- \frac{1}{2\pi i} \int_C R_g(\xi)\, u \; d\xi \quad \quad  u \in \Gamma(E_g).$$
By   Lemma \ref{resolvent} it follows that $F_g(C)\, u $
depends continuously on $g\in \mathcal V'$.\\

\emph{Step 2.} Let $d=\dim \mathbf F_{g_0} ( C)$ and $ u_{\lambda_1(g_0)} , \dots, u_{\lambda_d (g_0)}$ be the eigenfunctions of $P_{g_0}$ spanning $ \mathbf F_{g_0} ( C)$ with respective eigenvalues
$ \lambda_1(g_0) \leq \dots \leq \lambda_d (g_0)$.
Since $F_g(C) u$ depends continuously on $g \in \mathcal V'$, for all $u \in \Gamma(E_{g_0})$ 
we know that 
$$\lim_{g \to g_0} \left\|F_g(C)\,[  u_{\lambda_j(g_0)} ]-  u_{\lambda_j(g_0)} \right\|_0=0, \quad \text{for}\; j=1,\dots,d,$$
and therefore there exists $\mathcal V \subset \mathcal V'$ neighborhood of $g_0$ so that
 $$F_g(C)\,[u_{\lambda_1(g_0)}] ,\dots,F_g(C)\, [u_{ \lambda_d(g_0)}]$$ are linearly independent for $g \in \mathcal V$.
We thereby conclude, 
\begin{equation}\label{eq: step 2}
\dim \mathbf F_{g} ( C) \geq \dim \mathbf F_{g_0} ( C) \quad \quad \text{for}\;\; g \in \mathcal V.
\end{equation}

\emph{Step 3.} Now let $l:=\limsup_{g\to g_0} \dim \mathbf F_g(C)$. Consider $\{g_i\}_i \subset \mathcal V$ converging to $g_0$
so that $\dim {\mathbf F_{g_i}(C)}=l$ for $i=1,2,\dots$,  and let 
$\tau_{g_i} (u_{\lambda_{1}(g_i)}),\dots,\tau_{g_i} (u_{\lambda_{l}(g_i)})$ be the
eigensections that span $\mathbf F_{g_i}(C)$
with respective eigenvalues
$ \lambda_1(g_i) \leq \dots \leq \lambda_l (g_i)$. By Lemma \ref{fried},  
$\|\tau_{g_i}(u_{\lambda_{s}(g_{i})})\|_m \leq c_0 (|\lambda_{s}(g_i)|+1)$
is bounded for all $s=1,\dots,l$. Hence, since the Sobolev embedding is compact for $k>m$,  we may choose 
a subsequence $\{g_{i_h}\}_h$ for which $\{\tau_{g_{i_h}}( u_{\lambda_{s}(g_{i_h})})\}_h$ converges in the Sobolev norm $\|\cdot \|_k$
for all $s=1,\dots,l$.

Set $v_{s}:= \lim_h \tau_{g_{i_h}} (u_{\lambda_{s}(g_{i_h})})$ and observe that since $\tau_{g_0}$ is the identity,
$$P_{g_0} v_{s}=Q_{g_0} v_{s}= \lim_h Q_{g_{i_h}}  \tau_{g_{i_h}} (u_{\lambda_{s}(g_{i_h})})=\lim_h \lambda_{s}(g_{i_h}) \tau_{g_{i_h}} (u_{\lambda_{s}(g_{i_h})}).$$

It follows that $v_{1},\dots,v_{l}$ are linearly independent eigensections of $P_{g_0}$ 
that belong to $\mathbf F_{g_0}(C)$. Thereby, for $g \in \mathcal V$, equality \eqref{dim F} follows from inequality \eqref{eq: step 2} and
\[\dim \mathbf F_{g_0} ( C) \geq l=\limsup_{g\to g_0} \;\dim \mathbf F_g(C).\qedhere\]

\end{proof}

\subsection{Proof of Corollary \ref{main cor}} \label{proof main cor}\ \\

 For $\delta \in (0,1)$ and $\alpha \in (0, +\infty)$ with $\delta<\alpha$, consider  the sets
  \begin{align*}
  \mathcal G_{ \delta, \alpha}&:=\Big\{g \in \mathcal M: \;\,
 \lambda \text{ is \emph{simple} for all }   \lambda \in \Spec(P_g) \cap \left([-\alpha,-\delta ] \cup [\delta, \alpha ] \big.\right)
   \Big\}.
  \end{align*}

Assumming  the hypothesis of  Theorem \ref{main thm bundles} hold,
we prove in Proposition \ref{open dense} that
 the sets $\mathcal G_{ \delta, \alpha}$ are open and dense in $\mathcal M$ with the $C^\infty$-topology.\\
Let $\{\delta_k\}_{k \in \mathbb N}$
be a sequence in $(0,1)$ satisfying $\lim_{k}\delta_{k}=0$, and let
$\{\alpha_k\}_{k \in \mathbb N}$
be a sequence in $(0,+\infty)$ satisfying $\lim_{k}\alpha_{k}=+\infty$ and $\delta_k<\alpha_k$. Then
$$\bigcap_{k=1}^\infty \mathcal G_{\alpha_k, \delta_k}$$ is a residual set in $\mathcal M$
that coincides with the set of all Riemannian metrics for which all non-zero eigenvalues are simple.
For the proof of Corollary \ref{main cor} to be complete, it only remains to prove

\begin{proposition}\label{open dense}
Suppose that  
$P_g: \Gamma(E_g) \to \Gamma(E_g)$ has no rigid eigenspaces for a dense set of metrics.
Then, the sets  $\mathcal G_{ \delta, \alpha}$
 are open and dense in the $C^\infty$-topology.
\end{proposition}
\begin{proof}
We first show that $\mathcal G_{ \delta, \alpha}$ is open. Let $g_0 \in \mathcal G_{ \delta, \alpha}$ and
write $\lambda_1(g_0), \dots, \lambda_d(g_0)$ for all the eigenvalues of $P_{g_0}$ in $[-\alpha,-\delta ] \cup [\delta, \alpha ]$,
which by definition of $\mathcal G_{ \delta, \alpha}$ are simple.
Assume further that the eigenvalues are labeled so that
$$-\alpha \leq \lambda_1(g_0)<  \dots <  \lambda_k(g_0) \leq -\delta
\quad \text{and}\quad  \delta \leq \lambda_{k+1}(g_0) <\dots <  \lambda_d(g_0)\leq \alpha.$$
Consider $\ep_1>0$ small so that no eigenvalue of $P_{g_0}$ belongs to 
$$[-\alpha-\ep_1, -\alpha) \cup (-\delta, -\delta+\ep_1 ] \cup [\delta-\ep_1, \delta) \cup (\alpha, \alpha +\ep_1].$$

For all $ 1\leq i \leq k-1$ let $p_i:=\frac{1}{2}(\lambda_{i}(g_0)+ \lambda_{i+1}(g_0))$, and for  $k+2\leq i \leq d$
let $p_i:=\frac{1}{2}(\lambda_{i-1}(g_0)+ \lambda_{i}(g_0))$.
We also set $p_0:= -\alpha-\ep_1$, $p_k:=\delta+\ep_1$, $p_{k+1}:=\delta-\ep_1$ and $p_{d+1}:=\alpha+\ep_1$.\\

For all $ 1\leq i \leq k$ (resp. $k+1\leq i \leq d$), let $C_i$ be a differentiable curve that intersects the real axis transversally only at the
points $p_{i-1}$ and $p_{i}$  (resp. $p_i$ and $p_{i+1}$).
 In addition, let $\ep_2>0$ be so that for each $ 1\leq j \leq k-1$ and $k+2\leq j \leq d$, 
the circle $\hat C_j$ centered at $p_j$ of radius $\ep_2$ does not contain any eigenvalue of $P_{g_0}$.\\

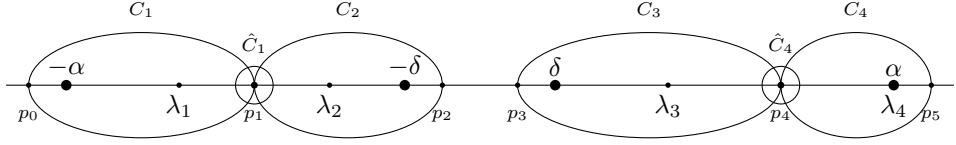
\begin{center}
  \begin{tikzpicture}

	 \draw[-] (-6.3,0) -- (6.3,0);
	  \node[above] at (-5.5, 0) {$-\alpha$};    \fill (-5.5,0) circle (2pt) ;
	   \node[above] at (-1, 0) {$-\delta$};    \fill (-1,0) circle (2pt) ;
          \node[above] at (1, 0) {$\delta$};    \fill (1,0) circle (2pt) ;
          \node[above] at (5.5, 0) {$\alpha$};   \fill (5.5,0) circle (2pt) ;

          \node[below] at (-6, -0.2) {$\scriptstyle p_0$};   \fill (-6,0) circle (1 pt) ;  
          \node[below] at (-3, -0.2) {$\scriptstyle p_1$};   \fill (-3,0) circle (1.3 pt) ;            
          \node[below] at (-0.5, -0.2) {$\scriptstyle p_2$};   \fill (-0.5,0) circle (1 pt) ;            
          \node[below] at (0.5, -0.2) {$\scriptstyle p_3$};   \fill (0.5,0) circle (1 pt) ;          
          \node[below] at (4, -0.2) {$\scriptstyle p_4$};   \fill (4,0) circle (1.3 pt) ;  
          \node[below] at (6, -0.2) {$\scriptstyle p_5$};   \fill (6,0) circle (1 pt) ;
          
          \node[below] at (-4, 0) {$\lambda_1$};   \fill (-4,0) circle (1 pt) ;        
          \node[below] at (-2, 0) {$\lambda_2$};   \fill (-2,0) circle (1 pt) ;  
          \node[below] at (2.5, 0) {$\lambda_3$};   \fill (2.5,0) circle (1 pt) ;
          \node[below] at (5.5, 0) {$\lambda_4$};   \fill (5.5,0) circle (1 pt) ;

    \draw (-3,0) circle (0.25cm);    \node[above] at (-3, 0.3) {$\scriptstyle {\hat C_1}$};
    \draw (4,0) circle (0.25cm);  \node[above] at (4, 0.3) {$\scriptstyle {\hat C_4}$};

    \draw (-1.75,0) ellipse (1.25cm and 0.7cm);   \node[above] at (-1.75, 0.78) {$\scriptstyle {C_2}$};
    \draw (-4.5,0) ellipse (1.5cm and 0.7cm);    	  \node[above] at (-4.5, 0.78) {$\scriptstyle {C_1}$};
    \draw (2.25,0) ellipse (1.75cm and 0.7cm); 	  \node[above] at (2.25, 0.78) {$\scriptstyle {C_3}$};
    \draw (5,0) ellipse (1 cm and 0.7cm); 	  \node[above] at (5, 0.78) {$\scriptstyle {C_4}$};

        \end{tikzpicture}
        \end{center}\ \\

By Proposition \ref{dimensions}, there exists an open  neighborhood $\mathcal V \subset \mathcal W_{g_0}$ of $g_0$ so that
for all $g \in \mathcal V$ and all $i,j$ for which $C_i$ and $\hat C_j$ were defined,
\begin{equation}\label{dim2}
\dim \mathbf F_g(C_i) = \dim \mathbf F_{g_0} ( C_i)=1
 \quad \;\; \text{and }\,  \quad \dim \mathbf F_g(\hat C_j) = \dim \mathbf F_{g_0} ( \hat C_j)=0.
 \end{equation}
Since $ [-\alpha,-\delta ] \cup [\delta, \alpha ]$ is contained in the union of all $C_i$'s and $\hat C_j$'s, 
it then follows from \eqref{dim} and \eqref{dim2} that for all $g \in \mathcal V$, 
$$\dim \Ker (P_g-\lambda I)=1  \quad \quad \forall \lambda \in \Spec(P_g) \cap \left( [-\alpha,-\delta ] \cup [\delta, \alpha ] \big.\right).$$
Since 
$\mathcal V \subset \mathcal G_{ \delta, \alpha}$, it follows that $\mathcal G_{ \delta, \alpha}$ is open.\\

We proceed to show that the sets $\mathcal G_{ \delta, \alpha}$ are dense. Let $g_0 \notin \mathcal G_{ \delta, \alpha}$ and
$\mathcal O$ be an open neighborhood of $g_0$.  Our assumptions imply that there exists $g \in \mathcal O$ so that the hypotheses  of Theorem \ref{main thm bundles} are satisfied for $P_g$.  It then follows that there exist a function $f \in C^\infty(M)$ so that
the metric $e^fg \in \mathcal O$ and all non-zero eigenvalues of $P_{e^fg}$ are simple. Therefore, $e^fg \in \mathcal O \cap \mathcal G_{ \delta, \alpha}$.

\end{proof}

\subsection{Proof of Corollary \ref{main cor}} \label{proof main cor rank}\ \\

 For $\delta \in (0,1)$ and $\alpha \in (0, +\infty)$, consider  the sets
  \begin{align*}
 \hat{\mathcal G}_{ \delta, \alpha}&:=\Big\{g \in \mathcal M: \quad 
   \dim \Ker (P_g-\lambda I) \leq rank (E_g)  \\
  &\hspace{4 cm}\text{ for all }\; \lambda \in \Spec(P_g) \cap \left([-\alpha,-\delta ] \cup [\delta, \alpha ] \big.\right)
 \Big\}
  \end{align*}

 Using the same argument in Proposition \ref{open dense} it can be shown that
 the sets $\hat {\mathcal G}_{ \delta, \alpha}$ are open. 
 To show that the sets $\hat {\mathcal G}_{ \delta, \alpha}$ are dense, one carries again
 the same argument presented in Proposition \ref{open dense}, using  the hypothesis 
of  Theorem \ref{main thm rank}  to find the metric $g$.
Let $\{\delta_k\}_{k \in \mathbb N}$
be a sequence in $(0,1)$ satisfying $\lim_{k}\delta_{k}=0$, and let
$\{\alpha_k\}_{k \in \mathbb N}$
be a sequence in $(0,+\infty)$ satisfying $\lim_{k}\alpha_{k}=+\infty$. Then
$\cap_k \hat{\mathcal G}_{\alpha_k, \delta_k}$ is a residual set in $\mathcal M$
that coincides with the set of all Riemannian metrics for which all non-zero eigenvalues of
$P_g$ have multiplicity smaller than the rank of the bundle $E_g$.
This completes the proof of Corollary \ref{main cor}.

\subsection {Proof of Theorem \ref{continuity}}\label{continuity subsec}\ \\
For $c\in \mathbb R$, we continue to write $\mathcal M_{c}=\{g \in \mathcal M:\; c \notin \Spec(P_g)\}$.
  In addition, for $g \in \mathcal M_c$, we write
$$ \mu_1(g) \leq \mu_2(g) \leq \mu_3(g) \leq \dots $$
 for the eigenvalues of $P_g$ in $(c, +\infty)$ counted with multiplicity.  We recall that it may happen that there are only finitely many
$\mu_{j}(g)$'s.
\\
To see that $\mathcal M_c$ is open, fix $g_0 \in \mathcal M_c$. Let $\delta>0$ be so that the circle $C_\delta$ centered at $c$ of radius $\delta$ contains no eigenvalue of $P_{g_0}$. By Proposition \ref{dimensions} there exists  $\mathcal V_1 \subset \mathcal W_{g_0}$, 
a neighborhood of $g_0$, so that for all $g \in \mathcal V_1$
\begin{equation*}
\dim \mathbf F_g(C_\delta) = \dim \mathbf F_{g_0} ( C_\delta)=0.
\end{equation*}
It follows that $\mathcal V_1 \subset \mathcal M_c$.\\

We proceed to show the continuity of the maps
$$\mu_i: \mathcal M_c \to \mathbb R, \qquad \qquad \quad g \mapsto  \mu_i(g).$$

We first show the continuity of $g \mapsto \mu_1(g)$ at $g_0 \in \mathcal M_c$. 
Fix $\ep_0>0$ and consider $0<\ep<\ep_0$ so that the circle $C_{\ep_1}$ centered at $\mu_1(g_0)$ 
of radius $\ep$  contains only the eigenvalue $\mu_1(g_0)$  among all the eigenvalues of $P_{g_0}$.
Let $\delta>0$ be so that there is no eigenvalue of $P_{g_0}$ in $[c-\delta, c]$. Consider a differentiable curve $C$
that intersects transversally the $x$-axis only at the points $c-\delta$ and $\mu_1(g_0)- \ep/2$.

By Proposition \ref{dimensions} there exists  $\mathcal V_2 \subset \mathcal W_{g_0}$, a neighborhood of $g_0$, so that for all $g \in \mathcal V_2$
\begin{equation*}
\dim \mathbf F_g(C) = \dim \mathbf F_{g_0} ( C) \quad \text{and} \quad \dim \mathbf F_g(C_{\ep}) = \dim \mathbf F_{g_0} ( C_{\ep}).
\end{equation*}
Since $\dim \mathbf F_{g_0} ( C)=0$, it follows that no perturbation $\mu_i(g)$, $i \geq 1$, belongs to $[c, \mu_1(g_0)- \ep/2]$.
Also, since the dimension of $ \mathbf F_g(C_\ep) $ is preserved, it follows that  there exists $j$ so that
 $|\mu_j(g)- \mu_1(g_0)|< \ep$ for $j\neq 1$. Since $$\mu_1(g_0)- \ep< \mu_1(g) \leq \mu_j(g) \leq \mu_1(g_0)+ \ep,$$ it follows that
for $g \in \mathcal V_2$ we have $|\mu_1(g)- \mu_1(g_0)|< \ep$ as wanted.

The continuity of $g \mapsto \mu_i(g)$, for $i \geq 2$, follows by induction. One should consider a circle of radius $\epsilon$
centered at $ \mu_i(g_0)$ and a  differentiable curve $C$
that intersects transversally the $x$-axis only at the points $c-\delta$ and $\mu_i(g_0)- \ep/2$.


\bibliographystyle{abbrv}
\bibliography{simpleev}
\end{document}